\documentclass[a4paper]{amsproc}
\usepackage{amssymb}
\usepackage[hyphens]{url} 
\usepackage[dvips]{graphicx}
\usepackage{amsmath}
\usepackage{cite}
\usepackage{mathrsfs}
\usepackage[title]{appendix} 
\usepackage[utf8]{inputenc}
\usepackage[english]{babel}
\usepackage{hyperref}
\RequirePackage{textgreek}

\hypersetup{linktoc=page}

 \theoremstyle{plain}
 \newtheorem{thm}{\bfseries Theorem}[section]
 \newtheorem{theorem}{\bfseries Theorem}
 
 \newtheorem{prop}[thm]{\bfseries Proposition}
 
 \newtheorem{observation}{\bfseries Observation}[section]
 \newtheorem{lem}[thm]{\bfseries Lemma}
 
 \newtheorem{cor}[thm]{\bfseries Corollary}
 
 \newtheorem{dfn}[thm]{\bfseries Definition}
 \theoremstyle{remark}
 \newtheorem{rem}[thm]{Remark}
 \numberwithin{equation}{thm}
 
 \def\lac{\overset{{\scriptstyle {\scriptscriptstyle \text{loc}}}}{\ll}}
 \def\lequ{\overset{{\scriptstyle {\scriptscriptstyle \text{loc}}}}{\sim}}

  \renewcommand{\leq}{\leqslant}
\renewcommand{\geq}{\geqslant}

\title[Equivalence-Singularity Dichotomy in Markov Measures]{Equivalence-Singularity Dichotomy in Markov Measures}

\subjclass[2010]{60J10, 28A35, 60F20}

\keywords{Markov chains, Markov fields, Kakutani dichotomy, zero-one law}

\author[Nachi Avraham-Re'em]{\bfseries Nachi Avraham-Re'em}

\address{
Einstein Institute of Mathematics \\
Edmond J. Safra Campus (Givat-Ram)\\
The Hebrew University of Jerusalem\\
9190401\\
Israel}

\email{nachman.avraham@mail.huji.ac.il}

\begin{document}

\vspace{18mm} \setcounter{page}{1} \thispagestyle{empty}

\begin{abstract}
We establish an equivalence-singularity dichotomy for a large class of one-dimensional Markov measures. Our approach is new in that we deal with one-sided and two-sided chains simultaneously, and in that we do not appeal to a 0-1 law. In fact we deduce a new 0-1 law from the dichotomy.
\end{abstract}

\maketitle

\tableofcontents

\section{Introduction}

A classical result by Kakutani asserts that any two locally-equivalent probability \emph{product} measures are either equivalent or mutually singular, and that there is a criteria for this dichotomy based on the convergence of a certain series of numbers \cite{kakutani}. The question of the extent to which this dichotomy holds in Markov chains was attributed to Vershik by Lodkin \cite{lodkin1971absolute}, and to the best of our knowledge the full answer is still unknown.

As Kakutani himself mentioned, such dichotomy is related to the tail 0-1 law. Therefore many authors established such dichotomies for other types of measures on product spaces under the assumption that the measures satisfy the tail 0-1 law; see \cite{lodkin1971absolute, lepage1972, lepage1975, ritter1979kakutani, kabanovlipcershiryaev1980, engelbertshiryaev, sato1991}. Also the well-known Feldman–Hájek dichotomy for Gaussian measures is a consequence of a certain 0-1 law; see \cite[$\mathsection6$ Theorem 5]{shiryaev1995p}. Indeed, using some kind of 0-1 law for this purpose seems natural as described by LePage \& Mandrekar: "\textit{Invariably, some type of 0-1 law is operative and suspected of forcing the equivalence-singularity dichotomy}." \cite{lepage1972}. Outside the scope of product spaces, there is a similar dichotomy for Riesz-product measures due to Brown \& Moran that does not make use of a 0-1 law. See \cite{brown1974orthogonality} and the references therein.

Shortly after Kakutani's work was published, it was discovered that one can describe the Hahn-Lebesgue decomposition of any pair of measures on a filtered space using martingales. We introduce this in detail in Section \ref{Section: Hahn-Lebesgue}. This decomposition was used thence to establish many criteria for equivalence and for singularity of measures in various contexts. Computing this decomposition for product measures reveals that it consists of a tail event, thus the Kakutani's dichotomy follows from the Kolmogorov 0-1 law. A little further effort, but yet elementary, is required to establish also the Kakutani's criteria; see for example
\cite[Theorem 4.3.8]{durrett2019probability}. 

Here we use this decomposition to establish an equivalence-singularity dichotomy for a large class of Markov measures, denoted by $\mathscr{S}$, without appealing to a 0-1 law. This class contains, among others, those Markov measures whose $M$-step transition probabilities stay away from zero for some positive integer $M$. See Definition \ref{Definition: Classes}. To prove the dichotomy for the class $\mathscr{S}$ we broaden our scope to (one-dimensional) \emph{Markov fields} rather than Markov chains, and this will enable us to deal also with two-sided Markov chains which were not considered by the previously mentioned authors. This gives us the following applications. First, we derive a new tail 0-1 law from our dichotomy for the class $\mathscr{S}$ (see Theorem \ref{Theorem C}). In particular, this shows that the Markov fields in $\mathscr{S}$ are in fact Markov chains (see Corollary \ref{Corollary: MFisMC}). Two more applications to ergodic theory will be given in Section \ref{Section: Applications}.

\subsubsection*{Outline of the Proof}
The computation of the Hahn-Lebesgue decomposition of Markov fields reveals that it consists of an event determined by the convergence of a random series with Markovian dependence. Then we show that for the Markov fields in $\mathscr{S}$ such series are well-behaved in the following sense. First, using some kind of probabilistic pigeonhole principle (Lemma \ref{Lemma: Fatou}) we are able to control the probability of such series to converge. Second, we exploit few class-properties of the Markov fields in $\mathscr{S}$ to use an exhaustion argument (Lemmas \ref{Lemma: Fatou Markov} and \ref{Lemma: Hereditary}).

\subsection*{Acknowledgment}

I thank my advisor, Zemer Kosloff, for his supportive guidance. I also thank Houcein El Abdalaoui and Yair Shenfeld for sending me suggestions and corrections, and the referee for their time and for their useful comments.

\section{Preliminaries and Results}

In this work all spaces are standard Borel spaces; all σ-algebras are contained in the Borel σ-algebra; all measures are Borel probability measures; and, all sets are Borel. We mostly follow the notations of \cite[$\mathsection$6]{shiryaev1995p}.

Let $\nu$ and $\mu$ be measures on a space $X$. We say that $\nu$ is \textit{absolutely continuous} with respect to $\mu$, and we write $\nu\ll\mu$, if $\mu\left(E\right)=0$ implies $\nu\left(E\right)=0$ for every set $E$ of $X$. If both $\nu\ll\mu$ and $\mu\ll\nu$ then we call $\nu$ and $\mu$ \textit{equivalent} and we write $\nu\sim\mu$. On the other extreme, we call $\nu$ and $\mu$ \textit{mutually singular}, and we write $\nu\perp\mu$, if there exists a set $E$ of $X$ such that $\nu\left(E\right)=0$ and $\mu\left(E\right)=1$.

A \textit{filtration} $\left(\mathcal{A}_n\right)_{n\geq0}$ of $X$ is an increasing sequence of σ-algebras that all together are generating the Borel σ-algebra. We follow the convention $\mathcal{A}_0=\left\{\emptyset,X\right\}$. Once we fixed a filtration $\left(\mathcal{A}_n\right)_{n\geq0}$ on $X$, for every measure $\nu$ on $X$ we abbreviate $\nu_{n}=\nu\mid_{\mathcal{A}_{n}}$ for $n\geq0$. A measure $\nu$  said to be \textit{locally absolutely continuous} with respect to another measure $\mu$, and we write $\nu\lac\mu$, if $\nu_{n}\ll\mu_{n}$ for every $n\geq0$. If both $\nu\lac\mu$ and $\mu\lac\nu$ we say that $\nu$ and $\mu$ are \textit{locally equivalent} and we write $\nu\lequ\mu$. Obviously, local absolute continuity is necessary for absolute continuity.

Let $\mathbb{I}$ stand either for the one-sided sequence of non-negative integers $\mathbb{Z}_{\geq0}$ or for the two-sided sequence of integers $\mathbb{Z}$. For a finite set $\mathcal{S}$ we consider the space $X=\mathcal{S}^{\mathbb{I}}$ with its usual Borel product σ-algebra. Let $\left(X_n\right)_{n\in\mathbb{I}}$ be the coordinate random variables of $X$, defined by $X_n\left(x\right)=x_n$ for $x=\left(x_k\right)_{k\in\mathbb{I}}\in X$. For a set $I\subset\mathbb{I}$ denote by $\sigma\left(X_i:i\in I\right)$ the σ-algebra generated by $\left\{X_i:i\in I\right\}$. The natural filtration $\left(\mathcal{A}_{n}\right)_{n\geq0}$ on $\mathcal{S}^{\mathbb{I}}$ is defined by $\mathcal{A}_n=\sigma\left(X_i:\left|i\right|\leq n\right)$.

For σ-algebras $\mathcal{A}$ and $\mathcal{A}'$ and a measure $\nu$ we write $\mathcal{A}\perp_{\nu}\mathcal{A}'$ if $\mathcal{A}$ and $\mathcal{A}'$ are independent with respect to $\nu$. A \textit{Markov chain} on $X$ is a measure $\nu$ on $X$ such that for every $n\in\mathbb{I}$,
\begin{equation}
\label{eq:1}\tag{$\mathbf{MC}$}
\sigma\left(X_{i}:i<n\right)\perp_{\nu}\sigma\left(X_{i}:i>n\right)\text{ conditioned on }\sigma\left(X_{n}\right).
\end{equation}
A \textit{Markov field} on $X$ is a measure $\nu$ on $X$ such that for every $n\geq1$,
\begin{equation}
\label{eq:2}\tag{$\mathbf{MF}$}
\sigma\left(X_{i}:\left|i\right|<n\right)\perp_{\nu}\sigma\left(X_{i}:\left|i\right|>n\right)\text{ conditioned on }\sigma\left(X_{-n},X_{+n}\right).
\end{equation}
In the one-sided case, the Markov chain property \ref{eq:1} and the Markov field property \ref{eq:2} coincide. In the two-sided case it is well-known that the Markov chain property \ref{eq:1} implies the Markov field property \ref{eq:2}, but the converse does not generally hold. See \cite[$\mathsection10$]{georgii2011gibbs} for the non-stationary case and \cite{chandgotia2014} for the stationary case.

We now define the classes of measures we discuss in this work.

\begin{dfn}
\label{Definition: Classes}
Let $X=\mathcal{S}^{\mathbb{I}}$ with the coordinate random variables $\left(X_n\right)_{n\in\mathbb{I}}$.
\begin{itemize}
    \item The class $\mathscr{R}$ consists of all measures $\nu$ on $X$ satisfying
$$\liminf_{\left|n\right|\to\infty}\nu\left(X_n=s\right)>0\text{ for all }s\in\mathcal{S}.$$
    \item The class $\mathscr{S}$ consists of all measures $\nu$ on $X$ satisfying
$$\liminf_{f\left(n,m\right)\to\infty}\nu\left(X_n=s,X_m=t\right)>0\text{ for all }s,t\in\mathcal{S},$$
where $f\left(n,m\right)=\min\left\{\left|n\right|,\left|m\right|,\left|n-m\right|\right\}$.
\end{itemize}
\end{dfn}

It is clear that $\mathscr{S}\subset\mathscr{R}$. Below we introduce a simple and concrete condition for a Markov field to be in the class $\mathscr{S}$ and, in particular, it will be clear that $\mathscr{S}$ contains all irreducible aperiodic stationary Markov chains (see Proposition \ref{Proposition: Doeblin}).

In order to discuss the one-sided case and the two-sided case simultaneously we use the following notations. Suppose that $\mathcal{S}$ is a finite set and consider the space $X=\mathcal{S}^{\mathbb{I}}$ with the coordinates variables $\left(X_n\right)_{n\in\mathbb{I}}$. Denote for $n\geq 1$,
$$\mathbb{S}=\begin{cases}
\mathcal{S} & \mathbb{I}=\mathbb{Z}_{\geq0}\\
\mathcal{S}\times\mathcal{S} & \mathbb{I}=\mathbb{Z}
\end{cases}\quad\text{and}\quad
\mathbb{X}_{n}=\begin{cases}
X_{n} & \mathbb{I}=\mathbb{Z}_{\geq0}\\
\left(X_{-n},X_{n}\right) & \mathbb{I}=\mathbb{Z}
\end{cases}.$$
In both cases we write $\mathbb{X}_0=X_0$.

For a Markov field $\nu$ on $X=\mathcal{S}^{\mathbb{I}}$ we denote by $\pi_0$ the initial distribution of $\mathbb{X}_0$ and the transition probabilities,
\begin{equation}
\label{eq:P_n}
P_{n}\left(s,t\right)=\nu\left(\mathbb{X}_{n+1}=t\mid\mathbb{X}_{n}=s\right)\text{ for }s,t\in\mathbb{S}\text{ and }n\geq1.
\end{equation}
Observe that in the one-sided case, $\pi_0$ together with $P_n$ for $n\geq1$ determine $\nu$ uniquely. However, in the two-sided case this is no longer true, as these transition probabilities only determine the joint distribution of $\left(X_{-\left(n+1\right)},X_{n+1}\right)$ conditioned on the joint distribution of $\left(X_{-n},X_{n}\right)$.

In the following Theorems \ref{Theorem A} and \ref{Theorem B} we consider a finite set $\mathcal{S}$ and a pair $\nu$ and $\mu$ of Markov fields on $\mathcal{S}^{\mathbb{I}}$, specified by $\left(\pi_0,P_n\right)_{n\geq1}$ and $\left(\lambda_0,Q_n\right)_{n\geq1}$, respectively. For such a pair let
\begin{equation}
\label{eq:D_n}
\mathrm{D}_{n}^{2}\left(\nu,\mu\right):=\sum_{s,t\in\mathbb{S}}\left(\sqrt{P_{n}\left(s,t\right)}-\sqrt{Q_{n}\left(s,t\right)}\right)^{2}\text{ for }n\geq1.
\end{equation}

The following Theorem \ref{Theorem A} generalizes \cite[Theorem 5.2]{avraham2022absolutely}.

\begin{theorem}[\textsc{A criteria for equivalence}]
\label{Theorem A}
Suppose that $\nu$ and $\mu$ are Markov measures such that $\nu\in\mathscr{R}$ and $\nu\lac\mu$. Then $\nu\ll\mu$ if and only if
$$\sum_{n\geq1}\mathrm{D}_{n}^{2}\left(\nu,\mu\right)<\infty.$$
In particular, if both $\nu,\mu\in\mathscr{R}$ and $\nu\lequ\mu$, then
$$\nu\ll\mu\iff\mu\ll\nu\iff\nu\sim\mu.$$
\end{theorem}

\begin{theorem}[\textsc{Equivalence-singularity dichotomy}]
\label{Theorem B}
Suppose that $\nu$ and $\mu$ are Markov measures in $\mathscr{S}$ and $\nu\lequ\mu$. Then either $\nu$ and $\mu$ are equivalent or that they are mutually singular. Moreover, $\nu\sim\mu$ if and only if
$$\sum_{n\geq1}\mathrm{D}_{n}^{2}\left(\nu,\mu\right)<\infty.$$
\end{theorem}

\section{Applications}
\label{Section: Applications}

We start by introducing a large class of Markov measures that belong to $\mathscr{S}$. For a Markov measure $\nu$ on $\mathcal{S}^{\mathbb{I}}$  consider the transition matrices $P_n\left(s,t\right)$ as defined in \ref{eq:P_n}. For $n\leq m$ denote $P^{\left(n,m\right)}=P_{n}\dotsm P_{m}$. We formulate the following proposition for the two-sided case. In the one-sided case it can be adapted with simple adjustments.

\begin{prop}
\label{Proposition: Doeblin}
Suppose that $\nu$ is a Markov measure on $\mathcal{S}^{\mathbb{Z}}$ with transition matrices $\left(P_n\right)_{n\in\mathbb{Z}}$ as in \ref{eq:P_n}. Then $\nu\in\mathscr{S}$ if the following conditions hold.
\begin{enumerate}
    \item There exists $0\leq\delta\leq 1/2$ such that
    $$P_{n}\left(s,t\right)=0\text{ or }P_{n}\left(s,t\right)\geq\delta\text{ for all }s,t\in\mathcal{S}\text{ and }n\in\mathbb{I}.$$
    \item There exists $M\geq 1$ such that $P^{\left(n,n+M\right)}$ is a positive matrix for all $n\in\mathbb{Z}$.
\end{enumerate}
\end{prop}

This type of condition is sometimes referred to as {\it Doeblin condition}. See for instance \cite{cohn1981paper}.

\begin{proof}
Observe that $P^{\left(n,n+M\right)}\geq\delta^{M}$ for all $s,t\in\mathcal{S}$ and $n\in\mathbb{Z}$. It follows that $P^{\left(n,n+N\right)}\geq\delta^{M}$ for all $s,t\in\mathcal{S}$ whenever $N\geq M$. Also, for all $t\in\mathcal{S}$ and $n\in\mathbb{Z}$,
$$\nu\left(X_{n}=t\right)=\sum_{s\in\mathcal{S}}\nu\left(X_{n-M}=s\right)P^{\left(n-M,n\right)}\left(s,t\right)\geq\delta^{M}.$$
Finally, by the Markov field property \text{\ref{eq:2}}, for all $s,t\in\mathcal{S}$ and $m-n\geq M$,
$$\nu\left(X_{n}=s,X_{m}=t\right)=P^{\left(n,m\right)}\left(s,t\right)\nu\left(X_{n}=t\right)\geq\delta^{2M}.\qedhere$$
\end{proof}

\subsection{The Double Tail 0-1 Law}

As we mentioned above, the proof we present for Theorem \ref{Theorem B} makes no reference to a 0-1 law. We rather deduce the double tail 0-1 law for the Markov measures in $\mathscr{S}$ from our dichotomy. Recall that the tail σ-algebra of $\mathcal{S}^{\mathbb{I}}$ with respect to the natural filtration $\left(\mathcal{A}_n\right)_{n\geq0}$ is given by
$$\mathcal{T}:=\bigcap_{n\geq1}\sigma\left(X_i:\left|i\right|>n\right).$$
We refer to $\mathcal{T}$ as \emph{one-sided tail} when $\mathbb{I}=\mathbb{Z}_{\geq 0}$ or \emph{double tail} when $\mathbb{I}=\mathbb{Z}$. Recall that the double tail is not only larger than each of the one-sided tails, from the left or from the right, but also larger then the σ-algebra generated by both the right one-sided tail and the left one-sided tail. See for instance \cite[Sections 2.5 and 2.6]{bradley2005basic} and the references therein. We start by introducing a useful terminology.

\begin{dfn}[\textsc{Kakutani Class}]
A class of measures on $\mathcal{S}^\mathbb{I}$ will be called a \emph{Kakutani class} if every pair of measures in it that are locally-equivalent are either equivalent or mutually singular.
\end{dfn}

The classical Kakutani's dichotomy, and in fact the Kolmogorov 0-1 law, asserts that the class of product measures is a Kakutani class. As we mentioned in the introduction, other Kakutani Classes are the Gaussian measures (by the Feldman–Hájek dichotomy), and a certain large class of Riesz-product measures (by the Brown--Moran dichotomy). It is well-known that the class of Markov chains on $\mathcal{S}^{\mathbb{N}}$ that satisfy the tail 0-1 law is a Kakutani class \cite{lodkin1971absolute,lepage1975,engelbertshiryaev}, and by Theorem \ref{Theorem: KLS} and Proposition \ref{Proposition: Local Kakutani-Hellinger} below the same holds for Markov fields on $\mathcal{S}^{\mathbb{I}}$. It turns out that sometimes the converse is also true, namely one can deduce the tail 0-1 law for the individual members of a Kakutani class.

To see this let us first discuss the Kakutani class of product measures. If $\nu$ is a product measure on $\mathcal{S}^{\mathbb{I}}$ and $T$ is a tail event (either one-sided tail or double tail) with $\nu\left(T\right)>0$, it is evident that the conditional measure $\nu_T\left(\cdot\right):=\nu\left(\cdot\mid T\right)$ is again a product measure that have the same marginals as $\nu$. Then clearly $\nu=\nu_T$ so that $\nu\left(T\right)=1$ and the Kolmogorov 0-1 law follows.

Observe that in this argument, the property that $\nu$ and $\nu_T$ have the same marginals can be replaced by the weaker property that $\nu\lequ\nu_T$, and then, as long as we know that $\nu_T$ is again a product measure, by Kakutani's dichotomy either $\nu\sim\nu_T$ or $\nu\perp\nu_T$. Since $\nu\left(T\right)>0$ necessarily $\nu\sim\nu_T$ so that $\nu\left(T\right)=1$ and the Kolmogorov 0-1 law follows. This argument can be formulated in general as follows.

\begin{observation}
\label{Observation}
Let $\mathscr{K}$ be a Kakutani class and $\mathcal{T}$ a σ-algebra satisfying
\begin{equation}
\label{eq:3}\tag{$\mathbf{H}$}
\forall\nu\in\mathscr{K},\quad\forall T\in\mathcal{T},\quad\text{if }\nu\left(T\right)>0\text{ then }\nu_T\in\mathscr{K}\text{ and }\nu_T\lequ\nu.
\end{equation}
Then every member of $\mathscr{K}$ takes only the values $0$ and $1$ on $\mathcal{T}$.
\end{observation}

The following Theorem \ref{Theorem C} generalizes the 0-1 law of Kosloff \cite[Theorem 3.7]{kosloff2019proving}.

\begin{theorem}[\textsc{Tail 0-1 law}]
\label{Theorem C}
Let $\mathcal{S}$ be a finite set. For the space $\mathcal{S}^{\mathbb{I}}$ let $\mathcal{T}$ be either the one-sided tail if $\mathbb{I}=\mathbb{Z}_{\geq0}$ or the double tail if $\mathbb{I}=\mathbb{Z}$. Then every Markov field in $\mathscr{S}$ satisfies the tail 0-1 law.
\end{theorem}

In view of Theorem \ref{Theorem B} and Observation \ref{Observation}, the proof of Theorem \ref{Theorem C} will follow once we show that the class of Markov fields in $\mathscr{S}$ satisfies the hereditary property \ref{eq:3} with respect to the tail σ-algebra. This will be shown in Section \ref{Section: Proofs}.

In \cite[Theorem 2.1]{zachary1983} (cf. \cite[Theorem 12.6]{georgii2011gibbs}), Zachary made the observation that Markov fields on $\mathcal{S}^{\mathbb{Z}}$ satisfy the following \emph{one-sided} property: For every $n\in\mathbb{Z}$ and $N\geq1$,
$$X_{n}\mid\sigma\left(X_{n-1},X_{n-2},\dotsc\right)\sim X_{n}\mid\sigma\left(X_{n-1},X_{n-N},X_{n-N-1},\dotsc\right).$$ This, together with the martingale convergence theorem, shows that a one-dimensional Markov field with trivial right-tail is a Markov chain. Then Theorem \ref{Theorem C} implies

\begin{cor}
\label{Corollary: MFisMC}
Every Markov field in $\mathscr{S}$ is in fact a Markov chain.
\end{cor}

\subsection{Ergodic Theory of the Symbolic Shift}

We introduce now two applications in ergodic theory. Recall that a measurable transformation of a measure space $S:\left(X,\nu\right)\to\left(X,\nu\right)$ is called \textit{measure-preserving} if $\nu=\nu_{*}S$ and is called \textit{non-singular} if $\nu\sim\nu_{*}S$, where $\nu_{*}S=\nu\circ S^{-1}$. Let $S:\mathcal{S}^\mathbb{I}\to \mathcal{S}^\mathbb{I}$ be the left-shift defined by $\left(Sx\right)_{n}=x_{n+1}$ for $n\in\mathbb{I}$. For every stationary Markov chain with an appropriate initial distribution the shift is measure-preserving, while for every non-stationary measure the shift is never measure-preserving but is possibly non-singular. It is a simple observation that the shift $S$ satisfies $\nu\in\mathscr{R}\iff\nu_{*}S\in\mathscr{R}$ and $\nu\in\mathscr{S}\iff\nu_{*}S\in\mathscr{S}$ for every Markov chain $\nu$. Another simple observation is that $\nu\lequ S_*\nu$ for a Markov field $\nu$ if, and only if, $\nu$ is supported positively on a subshift of finite type. That is, for all $s,t\in\mathcal{S}$, either $P_{n}\left(s,t\right)>0$ for all $n\in\mathbb{I}$ or $P_{n}\left(s,t\right)=0$ for all $n\in\mathbb{I}$. The following results first introduced in \cite{avraham2022absolutely} for a special case, and they are direct corollaries of Theorems \ref{Theorem A} and \ref{Theorem B}.

\begin{prop}
Let $\nu\in\mathscr{S}$ be a Markov field supported on a subshift of finite type in $\mathcal{S}^{\mathbb{I}}$. Then the shift is non-singular if, and only if,
$$\sum_{n\geq1}\sum_{s,t\in\mathbb{S}}\left(\sqrt{P_{n}\left(s,t\right)}-\sqrt{P_{n-1}\left(s,t\right)}\right)^{2}<\infty,$$
where $P_n$ are defined in \ref{eq:P_n}. Otherwise, the shift is totally-singular in the sense that $\nu\perp\nu_{*}S$.
\end{prop}

An important question in ergodic theory is whether a measure $\nu$ on $\mathcal{S}^\mathbb{I}$ admits an equivalent shift-invariant measure. For a survey including the topic see \cite{danilenko2011ergodic}. The question of whether there exists some measure, not necessarily Markovian, that is equivalent to $\nu$ and invariant to the shift, is hard and is still unsolved in its most general form. Under more assumption on the dynamical properties of the shift some solutions to this question were given in \cite{danilenkolemanczyk, avraham2022absolutely}. However, from Theorem \ref{Theorem B} we are able to determine whether a Markov measure $\nu$ on $\mathcal{S}^\mathbb{I}$ admits an equivalent shift-invariant \emph{Markov} measure.

\begin{prop}
Let $\nu\in\mathscr{S}$ be a Markov field on $\mathcal{S}^\mathbb{I}$. Then $\nu$ is equivalent to some shift-invariant Markov field if, and only if, the limit $P\left(s,t\right):=\lim_{\left|n\right|\to\infty}P_n\left(s,t\right)$ exists for all $s,t\in\mathcal{S}$ and moreover,
$$\sum_{n\geq1}\sum_{s,t\in\mathbb{S}}\left(\sqrt{P_{n}\left(s,t\right)}-\sqrt{P\left(s,t\right)}\right)^{2}<\infty,$$
where $P_n$ are defined in \ref{eq:P_n}. In this case, the stationary Markov field specified by $P$ is shift-invariant and equivalent to $\nu$. Otherwise $\nu$ is mutually singular with all stationary Markov field.
\end{prop}

\section{The Hahn-Lebesgue Decomposition of Markov Measures}
\label{Section: Hahn-Lebesgue}

Here we describe the Hahn-Lebesgue decomposition of a pair of locally-equivalent measures on a filtered space, and compute this decomposition for Markov fields.

Let $X$ be a space with filtration $\left(\mathcal{A}_n\right)_{n\geq0}$ and measures $\nu$ and $\mu$. If $\nu\lac\mu$ then the local Radon-Nikodym derivatives
$$\mathrm{z}_{n}:=\frac{d\nu_{n}}{d\mu_{n}}\in L^{1}\left(\nu\right),\quad n\geq1,$$
are well-defined and taking non-negative finite values. In fact, $\left(\mathrm{z}_{n},\mathcal{A}_{n}\right)_{n\geq1}$ is a martingale with respect to $\mu$ satisfying $\mathbf{E}_{\mu}\left(\mathrm{z}_n\right)=1$ for $n\geq1$, so by the martingale convergence theorem
$$\mathrm{z}_{\infty}:=\lim_{n\to\infty}\mathrm{z}_{n}\text{ exists }\mu\text{-a.e.}$$
It is also true that the same limit holds $\nu$-a.e. as well. However, this only means that $\mathrm{z}_{\infty}$ exists and is well-defined almost everywhere for each of $\nu$ and $\mu$, but possibly not on the same set. For a comprehensive representation see \cite[$\mathsection6$]{shiryaev1995p}.

The first part of the following Theorem \ref{Theorem: KLS} was introduced, shortly after Kakutani's work, independently by Pfeiffer following Kawada \cite{pfeiffer} and Grenander \cite[$\mathsection4$]{grenander1950stochastic}. Later it became a folklore and used by many others, among them Vershik \& Lodkin \cite{lodkin1971absolute}; LePage \& Mandrekar \cite{lepage1975}; and Kabanov, Lipcer \& Shiryaev \cite{kabanovlipcershiryaev1978, kabanovlipcershiryaev1980, shiryaev1978}. The second part of the following Theorem \ref{Theorem: KLS} was established by Kabanov, Lipcer \& Shiryaev \cite{kabanovlipcershiryaev1978, kabanovlipcershiryaev1980}. 

\begin{thm}
\label{Theorem: KLS}
In the above setting we have the following.
\begin{enumerate}
  \item The decomposition $X=\left\{\mathrm{z}_{\infty}=\infty\right\}\cup\left\{\mathrm{z}_{\infty}<\infty\right\}$  is the Hahn-Lebesgue decomposition of $\nu$ with respect to $\mu$, so that $\mu\left(\mathrm{z}_\infty<\infty\right)=1$ and
  $$\nu\left(E\right)=\nu\left(E\cap\left\{\mathrm{z}_{\infty}=\infty\right\} \right)+\intop_{E}\mathrm{z}_{\infty}\left(x\right)d\mu\left(x\right),\,\,\forall\text{ set }E\subset X.$$
  \item Let $\mathrm{Z}_{n}:=\mathrm{z}_{n}\mathrm{z}_{n-1}^{-1}$ where $\mathrm{Z}_{n}:=0$ if $\mathrm{z}_{n-1}=0$. Then
$$\left\{\mathrm{z}_{\infty}<\infty\right\}=\left\{ \sum_{n\geq1}\left(1-\mathbf{E}_{\mu}\left(\sqrt{\mathrm{Z}_{n}}\mid\mathcal{A}_{n-1}\right)\right)<\infty\right\}\mod\nu.$$
  \end{enumerate}
\end{thm}

Here and later we use $\mathbf{E}_{\mu}\left(\cdot\right)$ as a general notation for the expectation with respect to a measure $\mu$. By writing $A=B\mod\nu$ we mean that $\nu\left(A\triangle B\right)=0$.

The characterization in part (2) of Theorem \ref{Theorem: KLS} seems somehow mysterious. In the following we rephrase this criteria using the notion of the \textit{Kakutani-Hellinger distance}. Recall that for a pair of measures $\nu$ and $\mu$ on a measurable space $X$, their \textit{Kakutani-Hellinger distance} is defined by
$$\text{h}^2\left(\nu,\mu\right)=\left\Vert{\textstyle \sqrt{\frac{d\mu}{d\lambda}}-\sqrt{\frac{d\nu}{d\lambda}}}\right\Vert_{L^{2}\left(\lambda\right)}^{2}=\intop_{X}\left({\textstyle\sqrt{\frac{d\mu}{d\lambda}}-\sqrt{\frac{d\nu}{d\lambda}}}\right)^{2}d{\textstyle\lambda},$$
where $\lambda$ stands for any choice of a measure on $X$ for which both $\nu\ll\lambda$ and $\mu\ll\lambda$, for instance $\lambda=\left(\nu+\mu\right)/2$. It turns out that the value of $\text{h}^2\left(\nu,\mu\right)$ is independent of the choice of $\lambda$ and that $\left(\nu,\mu\right)\mapsto\text{h}\left(\nu,\mu\right)$ is a well-defined bounded metric of measures on $X$. See for instance \cite[Lemma 3.1]{shiryaev1995p}. If $\nu\ll\mu$ and $\mathrm{z}:=d\nu/d\mu$, the Kakutani-Hellinger distance takes the form
$$\mathrm{h}^{2}\left(\nu,\mu\right)=\mathbf{E}_{\mu}\left(1-\sqrt{\mathrm{z}}\right)^{2}=2\left(1-\mathbf{E}_{\mu}\left(\sqrt{\mathrm{z}}\right)\right).$$
For a detailed introduction of the Kakutani-Hellinger distance see \cite[$\mathsection$ 9]{shiryaev1995p}. Let us now 'localize' the Kakutani-Hellinger distance for measures on a filtered space.

\begin{dfn}[\textsc{Local Kakutani-Hellinger Distance}]
\label{Definition: local Kakutani-Hellinger distances}
Let $X$ be a space with filtration $\left(\mathcal{A}_n\right)_{n\geq0}$. For a measure $\nu$ on $X$ and $n\geq1$ denote by
$$\nu^{\left(n\right)}\left(E\mid x\right):=\mathbf{E}_{\nu_{n}}\left(\mathbf{1}_{E}\left(x\right)\mid\mathcal{A}_{n-1}\right),\quad E\in\mathcal{A}_{n},$$
the regular conditional probability, defined for $\nu$-a.e. $x\in X$ (cf. \cite[$\mathsection$ 7]{shiryaev1995p}).\\
For $\nu$ and $\mu$ measures on $X$ define the \emph{local Kakutani-Hellinger distances} by
$$\mathrm{h}_{n}^2\left(\nu,\mu\right)=\mathrm{h}^2\left(\nu^{\left(n\right)},\mu^{\left(n\right)}\right)\text{ for }n\geq 1.$$
For every $n\geq 1$ this is a $\mathcal{A}_{n-1}$-measurable random variables, as $\nu^{\left(n\right)}\left(\cdot\mid x\right)$ and $\mu^{\left(n\right)}\left(\cdot\mid x\right)$ are $\mathcal{A}_{n-1}$-measurable functions of $x$.
\end{dfn}

Observe that if $\nu\lequ\mu$, for every $n\geq 1$ we have that $\mathrm{z}_n:=d\nu_n/d\mu_n$ and $\mathrm{Z}_n:=\mathrm{z}_n\mathrm{z}_{n-1}^{-1}$ are strictly positive and finite, so that the following identity of $\mathcal{A}_{n-1}$-measurable functions holds:
$$\mathrm{Z}_{n}=\frac{d\nu_{n}}{d\mu_{n}}\cdot\frac{d\mu_{n-1}}{d\nu_{n-1}}=\frac{d\nu_{n}}{d\nu_{n-1}}\cdot\frac{d\mu_{n-1}}{d\mu_{n}}=\frac{d\nu^{\left(n\right)}}{d\mu^{\left(n\right)}}.$$
Thus, taking the measure $\lambda$ in the Kakutani-Hellinger distance to be simply $\mu$, we conclude
$$\mathrm{h}_{n}^{2}\left(\nu,\mu\right)=2\left(1-\mathbf{E}_{\mu}\left(\sqrt{\mathrm{Z}_{n}}\mid\mathcal{A}_{n-1}\right)\right)\text{ for all }n\geq1.$$
We then see that the criteria in Theorem \ref{Theorem: KLS} for absolute continuity of measures is a criteria on the divergence on the sum of the local Kakutani-Hellinger distances.

We now compute the Hahn-Lebesgue decomposition for Markov measures.

\begin{prop}
\label{Proposition: Local Kakutani-Hellinger}
Let $\nu$ and $\mu$ be Markov fields on $\mathcal{S}^{\mathbb{I}}$ specified by $\left(\pi_0,P_n\right)_{n\geq1}$ and $\left(\lambda_0,Q_n\right)_{n\geq1}$, respectively. Suppose that $\nu\lac\mu$ with respect to the natural filtration $\left(\mathcal{A}_n\right)_{n\geq0}$. Then for every $n\geq1$,
$$\mathrm{h}_{n}^{2}\left(\nu,\mu\right)=2\left(1-\mathbf{E}_{\mu}\left(\sqrt{\mathrm{Z}_{n}}\mid\mathcal{A}_{n-1}\right)\right)=\sum_{t\in\mathbb{S}}\left(\sqrt{P_{n}\left(\mathbb{X}_{n},t\right)}-\sqrt{Q_{n}\left(\mathbb{X}_{n},t\right)}\right)^{2}.$$
\end{prop}

\begin{proof}
The first equality is a general identity regarding the local Kakutani-Hellinger distances that we mentioned above. For the second equality, note that for every $n\geq1$ the local Radon-Nikodym derivative is
$$\mathrm{z}_{n}=\frac{d\nu_n}{d\mu_n}=\frac{\pi_0\left(X_{0}\right)\prod_{k=1}^{n-1}P_{k}\left(\mathbb{X}_{k},\mathbb{X}_{k+1}\right)}{\lambda_0\left(X_{0}\right)\prod_{k=1}^{n-1}Q_{k}\left(\mathbb{X}_{k},\mathbb{X}_{k+1}\right)},$$
so that
$$\mathrm{Z}_{n}=\mathrm{z}_{n}\mathrm{z}_{n-1}^{-1}=\frac{P_{n-1}\left(\mathbb{X}_{n-1},\mathbb{X}_{n}\right)}{Q_{n-1}\left(\mathbb{X}_{n-1},\mathbb{X}_{n}\right)}.$$
Then for every $n\geq1$,
\begin{align*}
\mathbf{E}_{\mu}\left(\sqrt{\mathrm{Z}_{n}}\mid\mathcal{A}_{n-1}\right)
&=\sum_{t\in\mathbb{S}}\sqrt{\frac{P_{n-1}\left(\mathbb{X}_{n-1},t\right)}{Q_{n-1}\left(\mathbb{X}_{n-1},t\right)}}Q_{n-1}\left(\mathbb{X}_{n-1},t\right)\\
&=\sum_{t\in\mathbb{S}}\sqrt{P_{n-1}\left(\mathbb{X}_{n-1},t\right)Q_{n-1}\left(\mathbb{X}_{n-1},t\right)}\\
&=1-\frac{1}{2}\sum_{t\in\mathbb{S}}\left(\sqrt{P_{n-1}\left(\mathbb{X}_{n-1},t\right)}-\sqrt{Q_{n-1}\left(\mathbb{X}_{n-1},t\right)}\right)^{2}.
\end{align*}
Rearranging terms we obtain the second equality.
\end{proof}

\section{Proofs}
\label{Section: Proofs}

We start to prove Theorem \ref{Theorem A}. Let $\nu$ and $\mu$ be Markov fields on $\mathcal{S}^{\mathbb{I}}$ specified by $\left(\pi_0,P_n\right)_{n\geq1}$ and $\left(\lambda_0,Q_n\right)_{n\geq1}$, respectively. Recall the number $\mathrm{D}_{n}^{2}\left(\nu,\mu\right)$ defined in \ref{eq:D_n}, and the local Kakutani-Hellinger distance $\mathrm{h}_{n}^{2}\left(\nu,\mu\right)$ as in Definition \ref{Definition: local Kakutani-Hellinger distances}. Using the formula computed in Proposition \ref{Proposition: Local Kakutani-Hellinger}, we see that
$$\mathrm{D}_{n}^{2}\left(\nu,\mu\right)\geq\mathrm{h}_{n}^{2}\left(\nu,\mu\right)\text{ for every }n\geq1.$$
Suppose now further that $\nu\in\mathscr{R}$. Since for every $n\geq 1$ we have
$$\frac{\mathbf{E}_{\nu}\left(\mathrm{h}_{n}^{2}\left(\nu,\mu\right)\right)}{\mathrm{D}_{n}^{2}\left(\nu,\mu\right)}=\frac{\sum_{s\in\mathcal{S}}\left(\sum_{t\in\mathcal{S}}\left(\sqrt{P_{n}\left(s,t\right)}-\sqrt{Q_{n}\left(s,t\right)}\right)^{2}\right)\pi_{n}\left(s\right)}{\sum_{s,t\in\mathcal{S}}\left(\sqrt{P_{n}\left(s,t\right)}-\sqrt{Q_{n}\left(s,t\right)}\right)^{2}},$$
it follows that
$$\liminf_{n\to\infty}\frac{\mathbf{E}_{\nu}\left(\mathrm{h}_{n}^{2}\left(\nu,\mu\right)\right)}{\mathrm{D}_{n}^{2}\left(\nu,\mu\right)}\geq\liminf_{n\to\infty}\inf_{s\in\mathcal{S}}\pi_{n}\left(s\right)>0.$$
Thus, by Theorem \ref{Theorem: KLS} and Proposition \ref{Proposition: Local Kakutani-Hellinger}, we obtain one of the implications in Theorem \ref{Theorem A}:

\begin{cor}
\label{Corollary: Half Theorem A}
If $\nu\in\mathscr{R}$ and $\nu\lac\mu$ then
$$\sum_{n\geq1}\mathbf{E}_{\nu}\left(\mathrm{h}_{n}^{2}\left(\nu,\mu\right)\right)<\infty\iff\sum_{n\geq1}\mathrm{D}_{n}^{2}\left(\nu,\mu\right)<\infty\implies\nu\ll\mu.$$
\end{cor}

For the other implication of Theorem \ref{Theorem A} we use the following probabilistic lemma that was proved in \cite[Appendix B]{avraham2022absolutely}. We also prove it here for completeness.

\begin{lem}
\label{Lemma: Fatou}
Let $I$ be a countable index set, $\left(a_{i}\right)_{i\in I}$ a sequence of non-negative numbers with $\sum_{i\in I}a_{i}=\infty$ and $\left(A_{i}\right)_{i\in I}$ a sequence of events in a probability space. Then
$$\mathbb{P}\left(\sum_{i\in I}a_{i}\mathbf{1}_{A_{i}}=\infty\right)\geq\liminf_{i\in I}\mathbb{P}\left(A_{i}\right).$$
\end{lem}

\begin{rem}[The Reverse Fatou's Lemma]
Taking $a_i=1$ for all $i\in I$, one concludes that for every $J$ an infinite subset of $I$,
$$\mathbb{P}\left(\sum_{i\in I}\mathbf{1}_{A_{i}}=\infty\right)\geq\mathbb{P}\left(\sum_{j\in J}\mathbf{1}_{A_{j}}=\infty\right)\geq\liminf_{j\in J}\mathbb{P}\left(A_{j}\right).$$
As $J$ is arbitrary this implies that $\mathbb{P}\left(\limsup_{i\in I}A_{i}\right)\geq\limsup_{i\in I}\mathbb{P}\left(A_{i}\right)$. This inequality is usually referred to as the {\it reverse Fatou's lemma}.
\end{rem}

\begin{proof}[Proof of Lemma \ref{Lemma: Fatou}]
Denote $p:=\liminf_{i\in I}\mathbb{P}\left(A_{i}\right)$. Excluding trivialities we assume $p>0$. For an arbitrary $0<q<p$ let $I_q\subset I$ be a co-finite set such that $\mathbb{P}\left(A_{i}\right)\geq q$ for every $i\in I_q$. For every $C>0$ let $F_{C}:=\left\{ \sum_{i\in I_q}a_{i}\mathbf{1}_{A_{i}}\leq C\right\}$. Suppose toward a contradiction that $\mathbb{P}\left(F_{C}\right)>1-q$ for some $C>0$. Then for every $i\in I_q$, since $\mathbb{P}\left(A_{i}^{\mathsf{c}}\right)\leq1-q$ we get
$$\mathbb{P}\left(F_{C}\right)-\mathbb{P}\left(A_{i}^{\mathsf{c}}\right)\geq\mathbb{P}\left(F_{C}\right)-\left(1-q\right)>0.$$
Thus,
\begin{align*}
\begin{split}
\infty
&=\sum_{i\in I_q}a_{i}\left(\mathbb{P}\left(F_{C}\right)-\left(1-q\right)\right)\leq\sum_{i\in I_q}a_{i}\left(\mathbb{P}\left(F_{C}\right)-\mathbb{P}\left(A_{i}^{\mathsf{c}}\right)\right)\\
&\leq\sum_{i\in I_q}a_{i}\mathbb{P}\left(F_{C}\cap A_{i}\right)=\sum_{i\in I_q}a_{i}\mathbf{E}\left({\mathbf{1}_{F_{C}}\mathbf{1}_{A_{i}}}\right)\\
&=\mathbf{E}\left(\mathbf{1}_{F_{C}}\sum_{i\in I_q}a_{i}\mathbf{1}_{A_{i}}\right)\leq C\mathbb{P}\left(F_{C}\right),
 \end{split}
\end{align*}
where we use that for general events $E$ and $F$ we have $\mathbb{P}\left(E\cap F\right)\geq\mathbb{P}\left(E\right)-\mathbb{P}\left(F^{\mathsf{c}}\right)$, monotone convergence, and the definition of $F_C$. This is a contradiction, so we conclude that $\mathbb{P}\left(F_{C}\right)\leq1-q$ for all $C\geq 0$, hence
$$\mathbb{P}\left(\sum_{i\in I}a_{i}\mathbf{1}_{A_{i}}=\infty\right)=\mathbb{P}\left(\sum_{i\in I_q}a_{i}\mathbf{1}_{A_{i}}=\infty\right)=\lim_{C\to\infty}\mathbb{P}\left(F_{C}^{\mathsf{c}}\right)\geq q.$$
Since $0<q<p$ is arbitrary the proof is complete.
\end{proof}

\begin{proof}[Proof of Theorem \ref{Theorem A}]
One of the implication of Theorem \ref{Theorem A} was proved in Corollary \ref{Corollary: Half Theorem A}. We show the other implication. If $\nu\sim\mu$ then by Theorem \ref{Theorem: KLS} and Proposition \ref{Proposition: Local Kakutani-Hellinger} we have
\begin{align*}
0
&=\nu\left(\sum_{n\geq1}\mathrm{h}_{n}^{2}\left(\nu,\mu\right)=\infty\right)\\
&=\nu\left(\sum_{n\geq1}\sum_{s,t\in\mathbb{S}}\left(\sqrt{P_{n-1}\left(s,t\right)}-\sqrt{Q_{n-1}\left(s,t\right)}\right)^{2}\mathbf{1}_{\left\{ \mathbb{X}_{n-1}=s\right\} }=\infty\right).
\end{align*}
For $s\in\mathbb{S}$ and $n\geq 1$ set
$$A_{n}\left(s\right)=\left\{\mathbb{X}_{n-1}=s\right\}$$
and
$$a_{n}\left(s\right)=\sum_{t\in\mathbb{S}}\left(\sqrt{P_{n-1}\left(s,t\right)}-\sqrt{Q_{n-1}\left(s,t\right)}\right)^{2}.$$
Assume by contradiction that $\sum_{n\geq1}a_{n}\left(s\right)=\infty$ for some $s\in\mathbb{S}$. By Lemma \ref{Lemma: Fatou},
$$0=\nu\left(\sum_{n\geq1}a_{n}\left(s\right)\mathbf{1}_{A_{n}\left(s\right)}=\infty\right)\geq\liminf_{n\to\infty}\nu\left(A_{n}\left(s\right)\right).$$
This contradicts the assumption $\nu\in\mathscr{R}$, hence $\sum_{n\geq1}a_{n}\left(s\right)<\infty$. This holds for all $s\in\mathbb{S}$ so that
$$\sum_{n\geq1}\mathrm{D}_{n}^{2}\left(\nu,\mu\right)=\sum_{n\geq1}\sum_{s\in\mathcal{S}}a_{n}\left(s\right)<\infty.\qedhere$$
\end{proof}

To prove Theorem \ref{Theorem B} we strengthen Lemma \ref{Lemma: Fatou} for Markov fields in $\mathscr{S}$.

\begin{lem}
\label{Lemma: Fatou Markov}
Let $\nu\in\mathscr{S}$ be a Markov field on $\mathcal{S}^{\mathbb{I}}$ and $\left(a_{n}\right)_{n\geq1}$ be a sequence of non-negative numbers for which $\sum_{n\geq1}a_{n}=\infty$. Then for every $s\in\mathbb{S}$,
$$\nu\left(\sum_{n\geq1}a_{n}\mathbf{1}_{\left\{\mathbb{X}_n=s\right\}}=\infty\right)=1.$$
\end{lem}

\begin{proof}[Proof of Theorem \ref{Theorem B} Assuming Lemma \ref{Lemma: Fatou Markov}]
Note that by Theorem \ref{Theorem A}, in order to establish Theorem \ref{Theorem B} to get a dichotomy it is enough to show that
$$\nu\not\sim\mu\Longrightarrow\nu\left(\sum_{n\geq1}\text{h}_{n}^{2}\left(\nu,\mu\right)=\infty\right)=1,$$
as the right-hand side is equivalent to $\nu\perp\mu$ by Theorem \ref{Theorem: KLS} and Proposition \ref{Proposition: Local Kakutani-Hellinger}. Assume that $\nu\not\sim\mu$. By Theorem \ref{Theorem A} it follows that $\sum_{n\geq1}\text{D}_{n}^{2}\left(\nu,\mu\right)=\infty$ so there exists $s\in\mathbb{S}$ for which $\sum_{n\geq1}a_{n}\left(s\right)=\infty$, where
$$a_{n}\left(s\right)=\sum_{t\in\mathbb{S}}\left(\sqrt{P_{n-1}\left(s,t\right)}-\sqrt{Q_{n-1}\left(s,t\right)}\right)^{2}\text{ for }n\geq1.$$
By Lemma \ref{Lemma: Fatou Markov} we conclude that
$$\nu\left(\sum_{n\geq1}\mathrm{h}_{n}^{2}\left(\nu,\mu\right)=\infty\right)\geq\nu\left(\sum_{n\geq1}a_{n}\left(s\right)\mathbf{1}_{\left\{\mathbb{X}_{n-1}=s\right\}}=\infty\right)=1.\qedhere$$
\end{proof}

Our strategy to prove Lemma \ref{Lemma: Fatou Markov} is as follows. For a measure $\nu$ on $\mathcal{S}^{\mathbb{I}}$ and a set $E$ with $\nu\left(E\right)>0$, let us denote by $\nu_E$ the conditional measure $\nu\left(\cdot\mid E\right)$. For every choice of $s\in\mathbb{S}$ and $A_n=A_n\left(s\right)=\left\{ \mathbb{X}_{n-1}=s\right\}$, let
$$E_\infty:=\left\{ \sum_{n\geq1}a_{n}\mathbf{1}_{A_{n}}=\infty\right\}.$$
Recall that Lemma \ref{Lemma: Fatou} dealt with an arbitrary measure, so we have that
$$\nu_E\left(E_\infty\right)\geq\liminf_{n\to\infty}\nu_{E}\left(A_n\right)\text{ for every set }E\text{ with }\nu\left(E\right)>0.$$
If we show that for every $E$ with $\nu\left(E\right)>0$ this limit-infimum is positive, that is $\nu_E\in\mathscr{R}$, by considering $E=E_{\infty}^{\mathsf{c}}$ it will simply follow that $\nu\left(E_\infty^{\mathsf{c}}\right)=0$ so $\nu\left(E_\infty\right)=1$. As we are about to show, this property indeed holds for Markov fields in $\mathscr{S}$, hence Lemma \ref{Lemma: Fatou Markov} is a consequence of the following Lemma \ref{Lemma: Hereditary}.

\begin{lem}
\label{Lemma: Hereditary}
Let $\nu\in\mathscr{S}$ be a Markov field on $\mathcal{S}^{\mathbb{I}}$. Then for every set $E$ with $\nu\left(E\right)>0$ it holds that $\nu_{E}\in\mathscr{S}\subset\mathscr{R}$.
\end{lem}

\begin{proof}
Since $\nu\in\mathscr{S}$ we have that
$$p:=\min_{s,t\in\mathbb{S}}\left\{ \liminf_{f\left(n,m\right)\to\infty}\nu\left(\mathbb{X}_{n}=s,\mathbb{X}_{m}=t\right)\right\}>0.$$
We first consider sets of some element $\mathcal{A}_{k}$ of the natural filtration. Let $E\in\mathcal{A}_{k}$ with $\nu\left(E\right)>0$. Then for every $t\in\mathbb{S}$ and $2k<n$ we have
\begin{align*}
\nu_{E}\left(\mathbb{X}_{n}=t\right)
&=\sum_{s\in\mathbb{S}\left(E,\left\lfloor n/2\right\rfloor\right)}\nu_{E}\left(\mathbb{X}_{\left\lfloor n/2\right\rfloor }=s\right)\nu\left(\mathbb{X}_{n}=t\mid E,\mathbb{X}_{\left\lfloor n/2\right\rfloor}=s\right)\\
\left(\text{\ref{eq:2}}\right)
&=\sum_{s\in\mathbb{S}\left(E,\left\lfloor n/2\right\rfloor\right)}\nu_{E}\left(\mathbb{X}_{\left\lfloor n/2\right\rfloor}=s\right)\nu\left(\mathbb{X}_{n}=t\mid\mathbb{X}_{\left\lfloor n/2\right\rfloor }=s\right)\\
&\geq\sum_{s\in\mathbb{S}\left(E,\left\lfloor n/2\right\rfloor\right)}\nu_{E}\left(\mathbb{X}_{\left\lfloor n/2\right\rfloor }=s\right)\nu\left(\mathbb{X}_{n}=t,\mathbb{X}_{\left\lfloor n/2\right\rfloor }=s\right),
\end{align*}
where for $j\geq1$ we write
$$\mathbb{S}\left(E,j\right)=\left\{ s\in\mathbb{S}:\nu\left(E,\mathbb{X}_{j}=s\right)>0\right\}.$$
We then see that
\begin{equation}
\label{eq:5}
\liminf_{n\to\infty}\nu_{E}\left(\mathbb{X}_{n}=t\right)\geq\liminf_{n\to\infty}\sum_{s\in\mathbb{S}\left(E,\left\lfloor n/2\right\rfloor \right)}\nu_{E}\left(\mathbb{X}_{\left\lfloor n/2\right\rfloor }=s\right)p=p>0.
\end{equation}
Now for $n,m$ with $k<n<m$ we have that
\begin{align*}
\nu_{E}\left(\mathbb{X}_{n}=s,\mathbb{X}_{m}=t\right)
&=\frac{\nu\left(\mathbb{X}_{n}=s,\mathbb{X}_{m}=t\right)}{\nu\left(E\right)}\nu\left(E\mid\mathbb{X}_{n}=s,\mathbb{X}_{m}=t\right)\\
\left(\text{\ref{eq:2}}\right)
&=\frac{\nu\left(\mathbb{X}_{n}=s,\mathbb{X}_{m}=t\right)}{\nu\left(E\right)}\nu\left(E\mid\mathbb{X}_{n}=s\right)\\
&=\nu\left(\mathbb{X}_{m}=t\mid\mathbb{X}_{n}=s\right)\nu_{E}\left(\mathbb{X}_{n}=s\right),
\end{align*}
and similar bound holds for $k<m<n$. Then using \ref{eq:5} we conclude that
$$\liminf_{f\left(n,m\right)\to\infty}\nu_{E}\left(\mathbb{X}_{n}=t,\mathbb{X}_{m}=t\right)\geq p^{2}>0,$$
and this completes the proof for sets of $\mathcal{A}_k$ for any $k\geq1$.

Let $E$ be a general set with $\nu\left(E\right)>0$, and some arbitrary $\epsilon>0$. Approximate $E$ by a set $E_0\in\mathcal{A}_k$ for some $k\geq1$, such that
$$\frac{\nu\left(E_{0}\right)}{\nu\left(E\right)}\geq1-\epsilon\quad\text{and}\quad\nu\left(E_{0}\backslash E\right)<\epsilon.$$
Then for every $n,m$ it holds that
\begin{align*}
\nu_{E}\left(\mathbb{X}_{n}=s,\mathbb{X}_{m}=t\right)
&\geq\nu_{E}\left(\mathbb{X}_{n}=s,\mathbb{X}_{m}=t,E_{0}\right)\\
&=\frac{\nu\left(E_{0}\right)}{\nu\left(E\right)}\nu_{E_{0}}\left(\mathbb{X}_{n}=s,\mathbb{X}_{m}=t\right)-\frac{\nu\left(\mathbb{X}_{n}=s,\mathbb{X}_{m}=t,E_{0}\backslash E\right)}{\nu\left(E\right)}\\
&\geq\left(1-\epsilon\right)\nu_{E_{0}}\left(\mathbb{X}_{n}=s,\mathbb{X}_{m}=t\right)-\frac{\epsilon}{\nu\left(E\right)}.
\end{align*}
Since $E_0\in\mathcal{A}_k$, by the first part of the proof we get
$$\liminf_{f\left(n,m\right)\to\infty}\nu_{E}\left(\mathbb{X}_{n}=s,\mathbb{X}_{m}=t\right)\geq\left(1-\epsilon\right)p^2-\frac{\epsilon}{\nu\left(E\right)}.$$
As $\epsilon>0$ is arbitrary we see that $\liminf_{f\left(n,m\right)\to\infty}\nu_{E}\left(\mathbb{X}_{n}=s,\mathbb{X}_{m}=t\right)\geq p^{2}>0$.
\end{proof}

We will now prove Theorem \ref{Theorem C}. In doing so we exploit the full generality of the Markov field property \ref{eq:2} upon the Markov chain property \ref{eq:1}. It is not hard to see, although we prove it in the course of the following proof, that when we condition a one-sided Markov chain on a one-sided tail event, the result is again a Markov chain. However, when we condition a two-sided Markov chain on a double tail event, the result may no longer be a Markov chain but rather a Markov field.

\begin{proof}[Proof of Theorem \ref{Theorem C}]
Recall that by Observation \ref{Observation} we need to show that the class of Markov fields that are in $\mathscr{S}$ is a Kakutani class that satisfies property \ref{eq:3}. That it is a Kakutani class is the essence of Theorem \ref{Theorem B}, so we will establish property \ref{eq:3}.

Let $\nu\in\mathscr{S}$ be a Markov field and let $T$ be a tail event, either  in the one-sided tail or in the double tail, with $\nu\left(T\right)>0$. We show that $\nu_T\in\mathscr{S}$, that $\nu_T$ is a Markov field and that $\nu_T\lequ\nu$. The fact that $\nu_T\in\mathscr{S}$ follows from Lemma \ref{Lemma: Hereditary} simply because $\nu\left(T\right)>0$. To see that $\nu_T$ is a Markov field, observe that if $E\in\sigma\left(\mathbb{X}_{k}:\left|k\right|>n\right)$ for some $n\geq1$,
\begin{align*}
\nu_{T}\left(E\mid\mathbb{X}_{n},\mathcal{A}_{n-1}\right)
&=\frac{\nu\left(E,T\mid\mathbb{X}_{n},\mathcal{A}_{n-1}\right)}{\nu\left(T\mid\mathbb{X}_{n},\mathcal{A}_{n-1}\right)}\\
\left(\text{\ref{eq:2}}\right)
&=\frac{\nu\left(E,T\mid\mathbb{X}_{n}\right)}{\nu\left(T\mid\mathbb{X}_{n}\right)}=\nu_{T}\left(E\mid\mathbb{X}_{n}\right),
\end{align*}
showing that $\nu_T$ is a Markov field. We then show that $\nu_T\lequ\nu$. Obviously $\nu_T\lac\nu$ so we show that $\nu\lac\nu_T$. Let us first consider the marginals. Suppose that $\nu\left(\mathbb{X}_{n}=s\right)>0$ for some $n\geq1$ and $s\in\mathbb{S}$. Fix some $t\in\mathbb{S}$ for which $\nu_T\left(\mathbb{X}_n=s,\mathbb{X}_{n+1}=t\right)>0$. Then
\begin{align*}
\nu_{T}\left(\mathbb{X}_{n}=s\right)
&\geq\nu_{T}\left(\mathbb{X}_{n}=s,\mathbb{X}_{n+1}=t\right)\\
&=\frac{\nu\left(\mathbb{X}_{n}=s,\mathbb{X}_{n+1}=t\right)}{\nu\left(T\right)}\nu\left(T\mid\mathbb{X}_{n}=s,\mathbb{X}_{n+1}=t\right)\\
\left(\text{\ref{eq:2}}\right)
&=\frac{\nu\left(\mathbb{X}_{n}=s,\mathbb{X}_{n+1}=t\right)}{\nu\left(T\right)}\nu\left(T\mid\mathbb{X}_{n+1}=t\right)\\
&=\frac{\nu\left(\mathbb{X}_{n}=s,\mathbb{X}_{n+1}=t\right)}{\nu\left(\mathbb{X}_{n+1}=t\right)}\nu_{T}\left(\mathbb{X}_{n+1}=t\right)\\
&=\nu\left(\mathbb{X}_{n}=s,\mathbb{X}_{n+1}=t\right)>0.
\end{align*}
Finally, let $E\in\mathcal{A}_{k}$ for some $k\geq1$ with $\nu\left(E\right)>0$. Fix some $s\in\mathbb{S}$ with $\nu\left(E,\mathbb{X}_{k+1}=s\right)>0$. In particular $\nu\left(\mathbb{X}_{k+1}=s\right)>0$ and by the equivalence of the marginals also $\nu_{T}\left(\mathbb{X}_{k+1}=s\right)>0$. It follows that
\begin{align*}
\nu_{T}\left(E\right)
&\geq\nu\left(T,E\mid\mathbb{X}_{k+1}=s\right)\nu\left(\mathbb{X}_{k+1}=s\right)\\
\left(\text{\ref{eq:2}}\right)
&=\nu\left(T\mid\mathbb{X}_{k+1}=s\right)\nu\left(E\mid\mathbb{X}_{k+1}=s\right)\nu\left(\mathbb{X}_{k+1}=s\right)\\
&=\nu_{T}\left(\mathbb{X}_{k+1}=s\right)\nu\left(E,\mathbb{X}_{k+1}=s\right)\frac{\nu\left(T\right)}{\nu\left(\mathbb{X}_{k+1}=s\right)}>0,
\end{align*}
showing that $\nu\lac\nu_T$.
\end{proof}

\bibliographystyle{acm}
\bibliography{References}

\nocite{*}

\end{document}